\documentclass{article}
\makeatletter
\let\@fnsymbol\@arabic
\makeatother
\usepackage{authblk}
\usepackage[english, activeacute]{babel} 
\usepackage{mathtools}
\usepackage{titling}
\usepackage{lipsum}
\usepackage{xcolor}
\usepackage[utf8]{inputenc}
\usepackage{listings}
\usepackage{tikz-cd}
\usepackage{graphicx}
\usepackage{amssymb}
\usepackage{qtree}
\usepackage{hyperref}
\usepackage{amsthm}
\usepackage{amsmath}
\usepackage{makeidx}
\usepackage{color}
\usepackage{enumerate}
\usepackage{doi}
\usepackage{mathdots}
\usepackage{xcolor}
\usepackage{biblatex}
\renewbibmacro{in:}{}
\usepackage{csquotes}
\usepackage[all,cmtip]{xy}

\newcommand{\comm}[1]{}
\usepackage{fancyhdr}
\newtheorem{teorema}{Theorem}

\newtheorem{lema}[teorema]{Lemma}
\newtheorem{proposicion}[teorema]{Proposition}

\theoremstyle{definition}

\newtheorem{definicion}[teorema]{Definition}

\newtheorem{remark}[teorema]{Remark}

\DeclareMathOperator{\PP}{\mathbb P}

\DeclareMathOperator{\Sing}{Sing}

\DeclareMathOperator{\linspan}{linspan}

\DeclareMathOperator{\rank}{rank}

\DeclarePairedDelimiter\ceil{\lceil}{\rceil}

\newtheorem*{theorem*}{Theorem}
\bibliography{main}

\title{Binary forms of suprageneric rank and the multiple root loci}

\author{Alejandro Gonz\'alez Nevado\thanks{Universität Konstanz, Germany - alejandro.gonzalez-nevado@uni-konstanz.de} \and Ettore Teixeira Turatti\thanks{DIMAI University of Florence, Italy - ettore.teixeiraturatti@unifi.it}}

\date{}

\begin{document}
\maketitle 

\begin{abstract}
We state the relation between the variety of binary forms of given rank and the dual of the multiple root loci. This is a new result for the suprageneric rank, as a continuation of the work by Buczyński, Han, Mella and Teitler. We describe the strata of these varieties and explore their singular loci.
\end{abstract}


\section{Introduction}

Let $V$ be a vector space of dimension $n+1$ over an algebraically closed field $K$ of characteristic zero. Let $f\in S_dV^*$ be a homogeneous form of degree $d$. The rank of $f$, also called the Waring rank, is defined to be the smallest integer $r$ such that $$f=l_1^d+\dots +l_r^d,$$ where $l_i,\ i=1,\dots,r$ are linear forms. 

The general rank $g$ of a form $f$, where the general rank means the rank that a general $f\in S_dV^*$ has, is a well known result and it is given by $$g=\ceil[\bigg]{\frac{\binom{n+d}{d}}{n+1}},$$ with exception of a finite number of cases, see \cite{Alexander} \cite{BRAMBILLA20081229}.

Let $S_{d,r}=\{f\in S_dV^*| \rank f=r\}$ be the set of forms of rank $r$. Let $X\subset \PP(S_d)$ be the Veronese variety. Then the variety obtained from the Zariski closure of $S_{d,r}$ coincides with the $r$-secant variety of the Veronese variety, $$
\sigma_r(X)=\overline{S_{d,r}}
$$
for every $r\leq g$. On the other hand, since $\sigma_g(X)$ fulfills the ambient space, $\overline{S_{d,r}}$ cannot be expressed as a secant variety if $r>g.$ The secant varieties of the Veronese variety have been vastly studied, but the last case has been less considered, however it has recently received more attention, as in \cite{seigal2019ranks}.

In this article we look at the case of binary forms of suprageneric rank, in the light of the work developed in \cite{gcms} on the strata of binary forms of rank at most generic, where the following description is obtained. 
\begin{theorem*}[Comas-Seiguer]
Let $0\leq k< \ceil{\frac{d+1}{2}}$ be an integer, then
$$
\overline{S_{d,k+1}}=(\cup_{i=1}^{k+1}S_{d,i})\bigcup(\cup_{i=0}^{k}S_{d,d-i+1}),
$$
where $\overline{S_{d,0}}=\overline{S_{d,d+1}}=\emptyset$. Furthermore, we have that $$\overline{S_{d,k+1}}\smallsetminus \overline{S_{d,k}}=S_{d,k+1}\cup S_{d,d-k+1}. $$
\end{theorem*}

We prove a similar result for the suprageneric case, notice that the second union has a shift on the indices.

\begin{teorema}\label{2}
Let $k$ be an integer and suppose that $d\geq d-k>\ceil{(\frac{d+1}{2})}$, then $\overline{S_{d,d-k}}$ is the union $$\overline{S_{d,d-k}}=(\cup_{i=1}^{k+1}S_{d,i})\bigcup (\cup_{i=0}^{k} S_{d,d-i}). $$ In particular $\overline{S_{d,d-k}}\smallsetminus \overline{S_{d,d-k+1}}=S_{d,k+1}\cup S_{d,d-k}$.
\end{teorema}

In order to prove these results, we show the relation between binary forms of fixed rank and the variety of multiple root loci. For an integer $0\leq r\leq \ceil[big]{\frac{n+1}{2}}$, it is known that $\overline{S_{d,r}}=\Delta_{2^r,1^{d-2r}}^\vee$. In \cite[Proposition 19]{jbkhmmzt} is proven that for the suprageneric case, $d\geq d-k>\ceil{\frac{d+1}{2}}$, $\overline{S_{d,d-k}}$ is the join of $k$ copies of the Veronese variety $X\subset \PP^d$ with its tangent variety $\tau(X)$, and compute the dimension $\dim(S_{d,d-k})=2k+2$, using this we find that an analogous relation with the dual of the multiple root loci also exist, and obtain the following proposition.

\begin{proposicion}\label{3}
Let $k$ be an integer and suppose that $d\geq d-k>\ceil{\frac{d+1}{2}},$ then $$\overline{S_{d,d-k}}=\Delta_{3,2^k,1^{d-2k-3}}^\vee.$$

\end{proposicion}

This proposition also allows us to prove that forms of rank different from $d-k$ in $S_{d,d-k}$ are singular points of this variety. More precisely 

\begin{teorema}\label{4}
Let $d-k>\ceil[big]{\frac{d+1}{2}}$, then the singular locus of $\overline{S_{d,d-k}}$ contains the subvariety $\overline{S_{d,k+1}}\cup\overline{S_{d,d-k+1}}.$ 
\end{teorema}

\noindent\textbf{Acknowledgements.} We would like to thank Giorgio Ottaviani for proposing this work and for the valuables discussions, suggestions and encouragement. This work has been generously supported by European Union's Horizon 2020 research and innovation programme under the Marie Skłodowska-Curie Actions, grant agreement 813211 (POEMA).

\section{Preliminaries}

\subsection{The Multiple Root Locus of Binary Forms}

We follow basically the notation in \cite{s}. Given an integer $n$, we say that a vector $\lambda=(\lambda_1,\dots,\lambda_d)$ is a partition of $n$ with $d$ parts if $\lambda_1\geq\dots\geq \lambda_d>0$ and $|\lambda|:=\lambda_1+\dots+\lambda_d=n$. Apart from this notation, we may also write a partition as a multiset $\lambda=\{1^{m_1},\dots,p^{m_p}\}$, where $m_i\geq 0$ is an integer for $i=1,\dots, p$, and represents that there are $m_i$ elements in the partition that are equal to $i.$ 

The set of homogeneous binary forms of degree $n$ corresponds to a variety on $\PP^n$ associating the points to the coefficient of each monomial in the polynomial expansion. The multiple root locus $\Delta_\lambda$ associated to a partition  $\lambda=(\lambda_1,\dots,\lambda_d)$ of $n$ is a subvariety of $\PP^n$ associated to the polynomials that have $d$ roots with multiplicity $\lambda_1,\dots,\lambda_d.$ The dimension of this variety is $\dim(\Delta_{\lambda})=d$ and its singular locus is a subset of the union $$\bigcup_{\lambda \mbox{ properly refines } \mu}\Delta_{\mu},$$ as described in \cite[Section 3]{sk}, and in \cite{jvc}.

We are particularly interested in the dual varieties $\Delta_{\lambda}^\vee.$ These are studied in \cite{Oeding_2012} and \cite{s}. In particular, Hilbert found that the degree of $\Delta_{\lambda}$ is $\deg(\Delta_{\lambda})=\frac{d}{m_{1}!\cdots m_{p}!}\lambda_{1}\cdots\lambda_{d}$ and, when the dual $\Delta_{\lambda}^{\vee}$ is a hypersurface (i.e., $m_{1}=0$), \cite[Theorem 5.3]{Oeding_2012} establishes that its degree is $\deg(\Delta_{\lambda}^{\vee})=\frac{(d+1)!}{m_{2}!\cdots m_{p}!}(\lambda_{1}-1)\cdots(\lambda_{d}-1).$

Notice that, given a partition $\lambda$ as above, we have another definition for $\Delta_\lambda$, it also is the image of $$
\xymatrix{ (\PP^1)^d \ar[r]& \PP^n,\ 
(l_1,\dots,l_d) \ar@{|->}[r] &l_1^{\lambda_1}\dots l_d^{\lambda_d}}.
$$ It follows that the dimension of $\Delta_\lambda$ is $d$ and its smooth points are those in which all the linear forms $l_i$ are pairwise different.

The following lemma gives an expression of the tangent space of a multiple root locus at a smooth point.

\begin{lema}\cite[Lemma 2.1, Lemma 2.2]{s} 
Let $f=l_1^{\lambda_1}\dots l_d^{\lambda_d}\in \Delta_\lambda$ be a smooth point and $g\in (\PP^n)^\vee.$ Then the tangent space at a $f$ is given by $$T_f\Delta_\lambda=\{h(x,y)\prod_{i=1}^d l_{i}^{\lambda_i-1}\mid h\in \PP(K[x,y]_d)\}.$$ Furthermore, $g\perp T_f\Delta_\lambda$ if and only if $$\prod_{i=1}^d l_i^{\lambda_i-1}(\frac{\partial}{\partial u},\frac{\partial}{\partial v})$$ annihilates $g(u,v)$.
\end{lema}

As we want to obtain a good description of the dual variety, we need to introduce the conormal variety. The conormal variety of $\Delta_\lambda$ is given by the closure of the set $$\{(f,g)\mid f\in\Delta_\lambda \text{ is a smooth point and } g\perp T_f\Delta_\lambda\}.$$ The dual variety $\Delta_\lambda^\vee$ is the image of the projection of the conormal variety onto the second factor. Therefore it is the variety of binary forms that are annihilated by some $f$. This observation leads to a parametrization of the conormal variety: it can be seen as the set of points $(f,g)$ of the form $$
f(x,y)=\prod_{i=1}^d(t_ix-s_iy)^{\lambda_i},\ g(u,v)=\sum_{i=1,\lambda_{i}\neq 1}^d(s_iu+t_iv)^{n-\lambda_i+2}g_i(u,v), 
$$ where $g_i(u,v)$ are binary forms of degree $\lambda_i-2$, and $(s_i,t_i)\in \PP^1$. The dimension of the dual variety to $\Delta_\lambda$ is given, using \cite[Corollary 7.3]{KATZ2003219}, by $$\dim\Delta_\lambda^\vee=n-m_1-1.$$

The inclusions between multiple root loci can be characterized in terms of refinements of the partitions that define them. Hence we have that $\Delta_\lambda\subset\Delta_\mu$ if and only if $\mu$ refines $\lambda$.

In addition, for a partition $\lambda=\{1^{m_1},\dots,p^{m_p}\}$, we denote its derived partition $\lambda':=(1^{m_2},\dots, (p-1)^{m_p})$, and this is a partition of $n-d$, where $d=\sum m_i$ is the number of parts. The next proposition gives a result similar to the one in the previous paragraph for inclusions between dual varieties. These inclusions are also characterized via refinements of partitions although it is not as direct as the previous one: the equivalent condition for the inclusion of duals involves refinements of derived partitions. Expressing this new condition requires thus the related partitions that we have just introduced.

\begin{proposicion}\cite[Proposition 3.4]{s}
Given two partitions $\lambda, \mu$ of $n$, then $\Delta_\lambda^\vee\subset\Delta_\mu^\vee$ holds if and only if $|\lambda'|\leq|\mu'|$ and, by adding to the parts, $\lambda'$ can be transformed into a partition $\tilde{\lambda}$ that is refined by $\mu'$. 
\end{proposicion}

\subsection{The Waring and the Forbidden Loci}

In this subsection we follow \cite{Carlini_2017}, where the notion of Waring locus and its counterpart, the forbidden locus, were introduced and we set $V$ here as a finite dimensional $\mathbb{K}$-vector space (not only of degree $2$). As before, let $SV$ be the vector space of forms over $V$ and $S_{d}V\subseteq SV$ the vector space of forms of degree $d,$ i.e., degree $d$ homogeneous polynomials over the field $\mathbb{K}$ in $\dim(V)$ variables. When there is no possibility of confusion we just write $S_{d}$ for $S_{d}V.$

\begin{definicion}
Given a form $f$ of degree $d$ and rank $r,$ the Waring locus $\mathcal{W}_f$ of $f$ is the set of linear forms that appear, up to scalar multiplication, in a minimal Waring decomposition of $f$, i.e., $$\mathcal{W}_f:=\{ [l]\in \mathbb{P}(S_1)\ \mid\ \exists l_2,\dots,l_r\in S_1,\  f\in\linspan(\{l^d,l_2^d,\dots,l_r^d\}) \}.
$$ The forbidden locus $\mathcal F_f$ is the complement, $\mathcal F_f:=\mathbb{P}(S_1)\smallsetminus\mathcal W_f$.
\end{definicion}

The Waring and the forbidden loci of binary forms are described in \cite[Theorem 3.5]{Carlini_2017}. Moreover, linear forms in these loci are important because it is easy to describe changes in the rank of forms $f$ when they are added to $\deg(f)$-th powers of linear forms lying on these loci. In particular, adding $\deg(f)$-th powers of linear forms in the forbidden locus never decreases the rank.

\begin{remark}
Notice that if a linear form $l\in S_{1}$ is an element of the forbidden locus $\mathcal F_f$ of the degree $d$ form $f$, then we have that $\rank(f+l^d)\geq \rank(f).$
\end{remark}

Now we introduce the apolar ideal, which will be fundamental together with the theorem associated to it to bridge multiple root loci and varieties generated by forms of certain fixed rank. The apolar ideal can be seen as the ideal formed by all the forms perpendicular to $f$ with respect to the scalar product by differentiation through dual variables.

\begin{definicion}
Let $f$ be a form of degree $d$, the apolar ideal of $f$, denoted $(f)^\perp$, is the ideal of elements $g\in SV^\vee$ such that $g\cdot f=0$, where $\cdot$ represents the contraction (by differentiation) of $f$ by $g$.
\end{definicion}

After this definition we remember the next well-known result which will be fundamental to bridge some multiple root loci with varieties generated by forms of fixed rank.

\begin{lema}[Apolarity Lemma]\label{al}
Let $f\in S_{d}.$ Then $f=l_{1}^{d}+\cdots+l_{s}^{d},$ where the summands $l_{i}\in S_{1}$ are pairwise non-proportional linear forms, if and only if $(f)^{\perp}\supseteq I,$ where $I$ is the ideal of the set $X=\{l_{1},\dots,l_{s}\}\subseteq S_{1}V^\vee$ of $s$ different points formed by all the $s$ pairwise non-proportional linear forms in the previous expression of $f$ as a sum of $d$-th powers of linear forms.
\end{lema}

A final important remark concerns the good description of the apolar ideal that we have in the case of a binary form. This description is fundamental in the proof of the result bridging multiple root loci and varieties generated by forms of fixed rank, which is itself an intermediate result towards our main (and final) theorem.

\begin{remark}
If $f$ is a binary form of degree $d,$ then $(f)^\perp=(g_1,g_2)$ with $\deg(g_1)+\deg(g_2)=n+2.$ In addition, if $\deg(g_1)\leq \deg(g_2),$ then $\rank(f)=\deg(g_1)$ if $g_1$ is squarefree and $\rank(f)=\deg(g_2)$ otherwise.
\end{remark}

\section{Binary forms of suprageneric rank}

\subsection{The variety of rank $k$ forms and the multiple root loci}

The relation between the variety $\overline{S_{d,k}}$ was well know for degrees smaller than $6$. So the first interesting example is the case where the degree is $d=6.$ We explore this case for ranks bigger than the generic rank $r=4.$ In the particular case of $f\in S_{6,6},$ we have that the ideal $(f)^\perp=(g_1,g_2)$ with $d_1+d_2=8,$ where $d_1$ and $d_2$ are the respective degrees. Since the rank of $f$ is $6$, we must have $d_1=2, d_2=6,$ and $g_1$ has a double root. Therefore the only possibility is that $g_1=l^2$, where $l$ is a linear form. In such case, by an immediate application of \cite[Lemma 2.2]{s}, we know that $f\in\Delta_{3,1^{3}}^\vee.$ The other inclusion follows from dimensional count. We can use such idea to compute any $S_{d,r}.$ For example, proceeding similarly for the rank $5$ we have that $d_1=3, d_2=5$ and therefore we have that $g_1$ has two possible cases: either $l_1^3$ or $l_1^2l_2$. In such case, $f\in\Delta_{4,1^{2}}^\vee$ or $f\in\Delta_{3,2,1}^\vee,$ respectively. We can see that the first is contained in the second, and therefore $f\in\Delta_{3,2,1}^\vee$. The other side follows again by dimensional count.

In \cite[Proposition 19]{jbkhmmzt} it was obtained that the dimension of $\overline{S_{d,r}}$ for $r$ bigger than the generic rank is given by $$\dim \overline{S_{d,r}}=\dim \sigma_{d-k+2}-1=2(d-r+1).$$ Using this fact together with the preceding idea developed in the example, we can obtain the following argument.

\begin{proof}[Proof of Proposition \ref{3}]
Let $f\in S_{d,d-k}$ be a homogeneous polynomial of degree $d$ and rank $d-k.$ We know that the apolar ideal $(f)^\perp$ is generated by $(g_1,g_2),$ such that $d_1+d_2=d+2$, with $d_1\leq d_2$ the respective degrees, and $\rank(f)=d_2,$ if $g_1$ is not squarefree, or $\rank(f)=d_1$ otherwise. So we may assume that $d_2=d-k, d_1=k+2$ and $g_1$ has a double root. Thence $g_1$ has the following form $l_0^2l_1\dots l_k$ and $f\in\Delta_{3,2^k,1^{d-2k-3}}^\vee.$ (Notice that all other possibilities for $g_1$, that is, with more than a single double root, lead to a different partition $\lambda$ but all of those are such that $\lambda'$ is refined by $(2,1^k)$ and therefore we have $\Delta_\lambda^\vee\subset \Delta_{3,2^k,1^{d-2k-3}}^\vee$ in such case.) It follows that $S_{d,d-k}\subseteq\Delta_{3,2^k,1^{d-2k-3}}^{\vee}$. On the other hand, by the proof of \cite[Proposition 19]{jbkhmmzt}, we have that $\dim \overline{S_{d,d-k}}=\dim \sigma_{k+2}-1=(2k+3)-1=2k+2,$ and $\dim\Delta_{3,2^k,1^{d-2k-3}}^\vee=d-m_1-1=2k+2$, so equality holds.
\end{proof}

Following \cite[Theorem 2]{gcms}, we obtain a similar result for the varieties of rank $r$ bigger than the generic rank. Furthermore, we also give another description for $\overline{S_{d,d-k}}$.

\begin{proof}[Proof of Theorem \ref{2}]
First we notice that, since $\overline{S_{d,d-k}}=\Delta_{3,2^k,1^{d-2k-3}}^\vee,$ then it can be described as the set of sums of powers of linear forms $$\{l_0^{d-1}g+l_1^d+\dots+l_k^d|l_i, g \text{ are linear forms, for }i=0,\dots k  \}.$$ 

It is trivial that $\cup_{i=0}^{k+1}S_{d,i}$ is contained on $\overline{S_{d,d-k}}$ by just considering $g=l_0$. If $f$ has rank $d-i$, with $i\leq k$, then $f=\sum^{d-i}_{j=1} l_j^d.$ It is possible to choose $l_{d-i+1},\dots, l_d$ linear forms such that $$f+l_{d-i+1}^{d}+\dots +l_d^d=\sum^{d}_{j=1} l_j^d$$ has rank $d$. This can be done by choosing the first linear form $l_{d-i+1}$ in the forbidden locus $\mathcal F_f$ of $f$, and repeating this procedure inductively to the new polynomial obtained. Since the obtained polynomial has rank $d$, we have that it can be written as $l^{d-1}g$, for some $l,g$ linear forms. Substituting it on the right side of the equation, we have that $$f=l^{d-1}g-\sum^{i}_{j=1} l_{d-i+j}^d,$$ which is an element of $\Delta_{3,2^k,1^{d-2k-3}}^\vee.$ So we obtained that $$\overline{S_{d,d-k}}\supset (\cup_{i=0}^{k+1}S_{d,i})\bigcup (\cup_{i=0}^{k} S_{d,d-i}).$$

For the other inclusion, suppose that $f\in \Delta_{3,2^k,1^{d-2k-3}}^\vee$, then $f=l_0^{d-1}g+l_1^d+\dots+l_k^d$ for some linear forms $l_0,\dots, l_{k},g$. We analyse two cases. If $g=l_0$, it is clear that $\rank(f)\leq k+1$. Otherwise, suppose that $g\neq l_0$, then $l_0^{d-1}g$ has rank $d$. Since all the other summands are power of linear forms, each of them can either decrease the rank by $1$, if they are on the Waring locus of $l_0^{d-1}g,$ or it does not change the rank, and it remains equal to $d$, if they are in the forbidden locus; in both cases, we have that $\rank(f)\geq d-k,$ hence we have the equality.
\end{proof}

\begin{proof}[Proof of Theorem \ref{4}]
Let $f=l_0^{d-1}g+l_1^d+\dots+l_k^d$ be a point of $\overline{S_{d,d-k}}$. We compute the tangent space at $f$, by considering $l_i=a_ix+b_iy$ and $g=\alpha x+\beta y$. We can consider a curve $$f(t)=\sum_{i=1}^k (a_i(t)x+b_i(t)y)^d+(a_0(t)x+b_0(t)y)^{d-1}(\alpha(t)x+\beta(t)y),$$ with $f(0)=f$, then taking the derivatives on the $a_i,b_i,\alpha,\beta$ we have that the tangent space is generated by $$
T_f\overline{S_{d,d-k}}=\langle yl_i^{d-1},xl_i^{d-1},xl_0^{d-2}g,yl_0^{d-2}g,xl_0^{d-1},yl_0^{d-1}\rangle, \ i=1,\dots,k.
$$

This space has $2k+4$ generators, but we notice that the last four of them span a $3$-dimensional space, so it has projective dimension $2k+2$ in a general point, as expected. We consider two cases now, first if $g$ is equal to $l_0$, in other words, the case that $f$ is a general element of $\overline{S_{d,k+1}}$. We notice that instead of a $3$-dimensional space, the last four elements on the span generates a $2$-dimenensional space, this means that the projective dimension of $T_f\overline{S_{d,d-k}}$ is at most $2k+1$, therefore $f$ is a singular point. Now instead,  assume that $l_i=l_j$ for some $i,j\neq 0$ and $i\neq j$, then $f$ is a general element of $\overline{S_{d,d-k+1}}$ and the dimension of $T_f\overline{S_{d,d-k}}$ is less than $2k+2$, again this gives that $f$ is a singular element of $\overline{S_{d,d-k}}$. 
\end{proof}
\subsection{The hypersurface $S_{2k+1,k+2}$}
Let $f\in \overline{S_{2k+1,k+2}}$, the maximal catalecticant matrix $C_f$ associated to $f$ has size $(k+1)\times (k+2)$.
In \cite[Theorem 4.1]{s} it is proven that this hypersurface has degree $2k(k+1)$ and its equation is computed, namely, the defining polynomial is the discriminant of
$$
q(u,v)=\det\begin{bmatrix}
u^{k+1}&u^{k}v&\dots& uv^{k}&v^{k+1}\\
a_0&a_1&\dots&a_{k}&a_{k+1}\\
a_1&a_2&\dots&a_{k+1}&a_{k+2}\\
\vdots&\vdots&\ddots&\vdots&\vdots\\
a_{k}&a_{k+1}&\dots&a_{2k}&a_{2k+1}
\end{bmatrix}.
$$
With this description we can obtain the following result.

\begin{teorema}
$S_{2k+1,k}$ is an irreducible component of $\Sing(S_{2k+1,k+2})$.
\end{teorema}

\begin{proof}
Let $(a_0,\ldots, a_d)$ be the coefficients of polynomials in $S_d$. The equation of $S_{2k+1,k+2}$ has degree $2k$ in the $(k+1)$-minors $b_j$ (for $j=0,\ldots, k$) of the maximal catalecticant matrix of size $(k+1)\times (k+2)$. Each $b_j$ is a homogeneous polynomial of degree $(k+1)$ in the $a_i$. Let $b_0^{\alpha_0}\ldots b_k^{\alpha_k}$ be a monomial with $|\alpha|=2k$. The derivative with respect to $a_i$ of such monomial 
is $\sum_j \alpha_j b_0^{\alpha_0}\ldots b_j^{\alpha_j-1}\ldots b_k^{\alpha_k}\frac{\partial b_j}{\partial a_i}$. Evaluated at a point $(a_0,\ldots, a_d)$ where all $b_j$ vanishes (this is a point in $S_{2k+1,k}$) this monomial vanishes. This concludes the proof.
\end{proof}

The case $k=2$ was studied before in \cite{pcgo} by Comon and Ottaviani, it is known as the \emph{apple invariant}, in such case the singular locus has two irreducible components, one is $S_{5,2}$, that comes from the minors of the catalecticant, and the other comes from the pullback from the locus of cubics with a triple root $\Delta_{3,1,1}$, that is the dual of the tangent { variety $\tau(S_{5,5})=S_{5,4}$}. For $k\geq 3$, $\Sing(S_{2k+1,k+2})$ has at least three irreducible components, one is $S_{2k+1,k}$, that is obtained from the minors of the catalecticant, the other two components arrives from the two irreducible components of the singular locus of the discriminant of $\sum_{i=0}^{k+1}a_it^i$, it comes as the pullback from the locus of degree $k+1$ polynomials with two double roots and with a triple root. For $k=3$ the components can be computed in Macaulay2, one is $S_{7,3}$, that has codimension $2$ and degree $10$. The other two components
have codimension $2$ and degree respectively $24$ (8 generators of degree 7, it comes as pullback from locus of quartics with two double roots)
and $36$ (55 generators of degree among 8 and 12, it comes as pullback from locus of quartics with a triple root), this case was named as the \emph{big apple invariant} in \cite{s}.
\printbibliography

\end{document}